\documentclass{siamart1116}
\usepackage{amsmath, amssymb, bm, mathrsfs}
\usepackage{cite, color, graphicx,subcaption,dsfont, accents}
\makeatletter
\def\widebar{\accentset{{\cc@style\underline{\mskip10mu}}}}
\makeatother

\newcommand{\re}{\mathop{\rm Re}\nolimits}

\newcommand{\dom}{\mathop{\rm dom}}

\makeatletter
\renewcommand{\theequation}{%
	\thesection.\arabic{equation}}
\@addtoreset{equation}{section}
\makeatother

\usepackage{lipsum}
\usepackage{amsfonts}
\usepackage{graphicx}
\usepackage{epstopdf}
\usepackage{algorithmic}
\ifpdf
  \DeclareGraphicsExtensions{.eps,.pdf,.png,.jpg}
\else
  \DeclareGraphicsExtensions{.eps}
\fi

\numberwithin{theorem}{section}

\newcommand{\TheTitle}{Stability Analysis of Perturbed Infinite-dimensional Sampled-data  Systems} 
\newcommand{\TheAuthors}{Masashi Wakaiki and Yutaka Yamamoto}

\headers{Perturbed Infinite-dimensional Sampled-data Systems}{\TheAuthors}

\title{{\TheTitle}\thanks{
\funding{This work was supported in part by JSPS KAKENHI Grant,
	JP17K14699 and JP19H02161.}}}

\author{
  Masashi Wakaiki\thanks{Graduate School of System Informatics, Kobe University, Nada, Kobe, Hyogo 657-8501, Japan
  	(\email{wakaiki@ruby.kobe-u.ac.jp}).} \and
  Yutaka Yamamoto\thanks{Professor Emeritus, Graduate School of 
  	Informatics, Kyoto University, Kyoto 606-8510, Japan
    ( \email{yy@i.kyoto-u.ac.jp}).}
}

\usepackage{amsopn}

\ifpdf
\hypersetup{
	pdftitle={\TheTitle},
	pdfauthor={\TheAuthors}
}
\fi

\newtheorem{example}[theorem]{Example}

\begin{document}

\maketitle

\begin{abstract}
	This paper addresses the stability analysis of
infinite-dimensional sampled-data systems
under unbounded perturbations.
We present two classes of unbounded perturbations preserving
the exponential stability of sampled-data systems.
To this end, we investigate the continuity of strongly continuous semigroups
with respect to their generators, considering the uniform operator topology.
\end{abstract}



\section{Introduction}
In this paper, we study the robustness of the exponential stability of infinite-dimensional sampled-data 
systems
under unbounded perturbations.
Modern control systems employ digital technology
for the implementation of controllers, and
sampled-data control plays an important role in cyber-physical systems.

It is also worth emphasizing that
the analysis and synthesis of sampled-data systems are of great theoretical interest due to
the interaction between continuous-time and discrete-time dynamics.
One particular approach
to overcome the difficulty arising from this hybrid property
is the lifting technique developed in \cite{Yamamoto1994, Bamieh1992},
which treats the intersample behavior of sampled-data systems in a unified, time-invariant framework.
For infinite-dimensional systems,
various sampled-data control problem have been studied,
for example, 
stabilization \cite{Tarn1988, Rebarber1998, Logemann2005, Logemann2013, Karafyllis2018,Kang2018Automatica},
robustness analysis with respect to sampling for stabilization \cite{Rebarber2002, Logemann2003, Rebarber2006},
and 
output regulation \cite{Logemann1997, Ke2009SCL, Ke2009SIAM, Ke2009IEEE, Wakaiki2019}.
However, relatively little work has been done on the stability analysis of
perturbed  infinite-dimensional sampled-data systems.

Consider the following sampled-data system with
state space $X$ and input space $U$ (both Banach spaces):
\begin{subequations}
	\label{eq:plant}
	\begin{align}
	\label{eq:state_equation}
	&\dot x (t) = Ax(t) +B_0u(t)\quad \forall t \geq 0; \qquad 
	x(0) = x^0 \in X, \\
	\label{eq:input}
	&u(k\tau + t) = F x(k\tau)\qquad \forall t \in[0, \tau),~\forall k \in \mathbb{Z}_+,
	\end{align}
\end{subequations}
where  $x(t) \in X$, $u(t) \in U$, $\tau>0$ is a sampling period, 
$A$ is the generator of a strongly continuous semigroup 
$T(t)$ on $X$, and the input operator $B_0$ 
and the feedback operator $F$ 
are bounded linear operators from $U$ to $X$ and from $X$ to $U$, respectively.
Since we here give only stability analysis and do not go over to feedback design,
we define the bounded operator $B := B_0F$ on $X$ for simplicity of notation.
By the standard theory of strongly continuous semigroups (see, e.g., \cite{Hille1957, Curtain1995, Engel2000, Arendt2001}), 
this abstract evolution equation \eqref{eq:plant}
has the unique solution defined recursively by
\begin{subequations}
	\label{eq:unique_solution}
	\begin{align}
	x(0) &= x^0, \\
	x(k\tau+t) &= T(\tau) x(k\tau) + \int^{t}_0 T(s)Bx(k\tau) ds\quad 
	\forall t \in (0, \tau],~\forall k \in \mathbb{Z}_+.
	\end{align}
\end{subequations}
We assume that 
the nominal sample-data system \eqref{eq:plant} is exponentially stable, which is defined as follows:
\begin{definition}[Exponential stability]
	The sampled-data system \eqref{eq:plant} is exponentially stable with  decay rate 
	greater than $\omega$ if 
	there exist $M \geq 1$ and $\widetilde \omega >\omega$ such that 
	the solution $x(t)$ given by \eqref{eq:unique_solution} satisfies
	\begin{equation}
	\label{eq:ex_stability}
	\|x(t)\| \leq M e^{-\widetilde \omega t} \|x^0\| \qquad 
	\forall x^0 \in X,~\forall t \geq 0. 
	\end{equation}
\end{definition}
Consider a linear operator $D:\dom(D) \subset X \to X$, and
let the generator $A$ in \eqref{eq:state_equation} be perturbed by this operator $D$.
Consequently, the perturbed sampled-data system is written as
\begin{subequations}
	\label{eq:perturbed_plant}
	\begin{align}
	\label{eq:perturbed_state_equation}
	&\dot x (t) = (A+D)x(t) +B_0u(t)\quad \forall t \geq 0; \qquad 
	x(0) = x^0 \in X, \\
	\label{eq:perturbed_input}
	&u(k\tau + t) = F x(k\tau)\qquad \forall t \in[0, \tau),~\forall k \in \mathbb{Z}_+.
	\end{align}
\end{subequations}
The difficulty in analyzing the stability of the perturbed system \eqref{eq:perturbed_plant} 
arises from the unboundedness of the operator $D$.
In this paper, we focus on two typical classes of unbounded perturbations and 
provide conditions for preservation of exponential stability.
We denote by $X_1$ the Banach space $\big(\hspace{-1pt}\dom (A), \|\cdot\|_A\big)$, where $\|\cdot\|_A$ is the graph norm
for $A$, i.e., $\|x\|_A := \|x\|+\|Ax\|$  for $x \in \dom (A)$.
The first theorem shows  that the exponential stability of the sampled-data system is robust
against unbounded perturbations called Miyadera-Voigt perturbations.
\begin{theorem}
	\label{thm:DT_case}
	Suppose that the nominal sampled-data system \eqref{eq:plant}
	is exponentially stable with decay rate greater than $\omega$.
	Choose $t_0>0$ arbitrary.
	There exists $0< q < 1$ such that for every bounded linear operator $D:X_1 \to X$ satisfying 
	\[
	\int^{t_0}_0 \|DT(s)x\|ds \leq q\|x\|\qquad \forall x \in \dom(A),
	\]
	the perturbed sampled-data system  \eqref{eq:perturbed_plant} is also
	exponentially stable with decay rate greater than $\omega$.
\end{theorem}

In the second theorem, we consider only the case where $T(t)$ is an analytic semigroup (see, e.g.,
Chapter~XVII in \cite{Hille1957},
Section~II.4.a in \cite{Engel2000}, and Section~3.7 in \cite{Arendt2001}
for analytic semigroups).
This restriction  allows us to deal with
a larger class
of unbounded perturbations, called 
relatively $A$-bounded perturbations, than in the first theorem.
\begin{theorem}
	\label{thm:RB_case}
	Suppose that the nominal sampled-data system \eqref{eq:plant}
	is exponentially stable with decay rate greater than $\omega$ and that 
	$T(t)$ is an analytic semigroup. 
	There exist $\alpha,\beta>0$ such that for every linear operator $D:\dom(D) \subset X \to X$ 
	satisfying $\dom(A) \subset \dom(D)$ and
	\[
	\|Dx\| \leq \alpha \|Ax\| + \beta \|x\|\qquad \forall x \in \dom(A),
	\] 
	the perturbed sampled-data system  \eqref{eq:perturbed_plant} is also
	exponentially stable with decay rate greater than $\omega$.
\end{theorem}
Note that the theorems above implicitly assume the existence of the strongly continuous
semigroup generated by the perturbed operator $A+D$.
The existence of the semigroup has been proved 
in \cite{Miyadera1966} for Miyadera-Voigt perturbations
and in \cite[Chapter~XIII]{Hille1957} for relatively $A$-bounded perturbations, respectively; 
see Theorems~\ref{thm:MV_perturbation_case} and
\ref{thm:perturbed_analytic_semigroup}
below for more rigorous statements.

Let $T_D(t)$ denote the strongly continuous semigroup generated by $A+D$.
In the beginning of the proofs of  Theorems~\ref{thm:DT_case} and \ref{thm:RB_case},
we show that if
$\|T(t)-T_D(t)\|$ is sufficiently small on $[0,\tau]$, then the perturbed sampled-data system \eqref{eq:perturbed_plant}
is also
exponentially stable.
To evaluate $\|T(t)-T_D(t)\|$ on $[0,\tau]$, we use the variation of parameter formula for 
Miyadera-Voigt perturbations
and the integral representation of an analytic semigroup for relatively $A$-bounded
perturbations, respectively.

In Section~2,
we present  basic results on the exponential stability of
sampled-data systems and the robustness of power stability.
Section~3 is devoted to the proof of Theorem~\ref{thm:DT_case}.
In Section~4,  we  prove Theorem~\ref{thm:RB_case} and then  provide an example to compare the perturbation classes in
Theorems~\ref{thm:DT_case} and \ref{thm:RB_case}.
In Section~5,
we present several examples to illustrate the obtained results.
\paragraph*{Notation and terminology}
We denote by $\mathbb{Z}_+$ the set of nonnegative integers.
For $r >0$, we define 
$\mathbb{D}_r := \{z \in \mathbb{C}:
|z | < r
\}$ 
and 
$\mathbb{E}_r := \{z \in \mathbb{C}:
|z | > r
\}$.
For a set $\Omega \subset \mathbb{C}$, its closure is denoted by $\overline {\Omega}$.
For $\delta \in (0,\pi/2]$, we define
$
\Sigma_{\delta} :=
\left\{
z \in \mathbb{C} \setminus \{0\}:|\arg z| \leq \delta
\right\}.
$
Let $X$ and $Y$ be Banach spaces. Let us denote by $\mathcal{L}(X,Y)$  the space of
all bounded linear operators from $X$ to $Y$. We write $\mathcal{L}(X)$ for $\mathcal{L}(X,X)$.
An operator $\Delta \in \mathcal{L}(X)$ is said to be {\em power stable}
if there exists constants $M \geq 1$ and $\theta \in (0,1)$ such that $\|\Delta^k\| \leq M \theta^k$ for
every $k \in \mathbb{Z}_+$.
For a linear operator $A$ from $X$ to $Y$, we denote by $\dom (A)$ the domain of $A$.
The resolvent set of a linear operator $A:\dom(A) \subset X \to X$ is denoted by $\varrho(A)$, and
$R(\lambda,A) := (\lambda I - A)^{-1}$ for $\lambda \in \varrho(A)$.

\section{Preliminaries}

For the nominal sampled-data system \eqref{eq:plant},
define the operator $\Delta_\tau \in \mathcal{L}(X)$  by
\begin{equation}
\label{eq:Delta_def}
\Delta_\tau := T(\tau) + \int^\tau_0 T(s)B ds.
\end{equation}
The state $x$ defined by \eqref{eq:unique_solution}
satisfies $x\big((k+1)\tau\big) = \Delta_\tau x(k\tau)$ for every $k \in \mathbb{Z}_+$.

Lemma~\ref{lem:ex_po_stability} below provides the relationship between
the exponential stability of the sampled-data system 
and the power stability of $\Delta_\tau$. One 
can obtain this result by slightly modifying the proof
of Proposition 2.1 in \cite{Rebarber1998}.
\begin{lemma}
	\label{lem:ex_po_stability}
	For any $\tau>0$, the sampled-data system \eqref{eq:plant} is exponentially stable with decay rate greater than $\omega$ 
	if and only if $e^{\omega \tau}\Delta_\tau$ is power stable.
\end{lemma}

By Corollary~4.5 of \cite{Wirth1994},
which gives the stability radius of an infinite-dimensional discrete-time system,
one can also immediately obtain the following lemma:
\begin{lemma}
	\label{lem:power_stable_perturb}
	Let $X$ be a Banach space, and let $\Delta_1 \in \mathcal{L}(X)$ and $\kappa >0$. If 
	$\kappa \Delta_1 $ is power stable, then
	there exists $\epsilon >0$ such that $\kappa \Delta_2$ is also power stable for
	every $\Delta_2 \in \mathcal{L}(X)$ satisfying 
	$\|\Delta_1 - \Delta_2\|< \epsilon$.
\end{lemma}
%

Applying Lemma~\ref{lem:power_stable_perturb} to the operator $\Delta_\tau$ defined by
\eqref{eq:Delta_def}, we obtain the following simple result.
\begin{proposition}
	\label{lem:semigroup_epsilon}
	For $j \in \{1,2\}$,
	let $T_j$ be strongly continuous semigroups and define
	\[
	\Delta_{\tau,j} := T_j(\tau) + \int^\tau_0 T_j(s)B ds.
	\]
	For a given $\kappa >0$,
	suppose that $\kappa \Delta_{\tau,1} $ is power stable. 
	There exists $\epsilon>0$ such that 
	the inequality
	\begin{equation}
	\label{eq:T1T2_diff}
	\|T_1(t) - T_2(t)\| < \epsilon \qquad \forall t \in [0,\tau]
	\end{equation}
	implies that $\kappa \Delta_{\tau,2} $ is also power stable.
\end{proposition}
\begin{proof}
	Suppose that \eqref{eq:T1T2_diff} holds for some $\epsilon>0$. 
	Since
	\[
	\Delta_{\tau,1} - \Delta_{\tau,2} = 
	T_1(\tau) - T_2(\tau) 
	+
	\int^\tau_0 \big[
	T_1(s) - T_2(s)
	\big] B ds,
	\]
	it follows that 
	\begin{align*}
	\|\Delta_{\tau,1} - \Delta_{\tau,2} \|
	&\leq 
	\|T_1(\tau) - T_2(\tau)\| + 
	\int^\tau_0	\|T_1(s) - T_2(s)\| \cdot 
	\| B\| ds \\
	&\leq \epsilon
	\big(
	1 + \tau \|B\|
	\big).
	\end{align*}
	Thus, Lemma~\ref{lem:power_stable_perturb} yields the desired conclusion.
\end{proof}

Motivated by Lemma~\ref{lem:ex_po_stability} and 
Proposition~\ref{lem:semigroup_epsilon}, 
we will study  the continuity of strongly continuous semigroups 
with respect to their generators in Sections~3 and 4.
For the strong operator topology, this continuity has been obtained as the Trotter-Kato Approximation theorem 
in \cite{Trotter1958, Kato1959};
see also Section~III.4.b in \cite{Engel2000} and Section 3.6 in \cite{Arendt2001}.
However, we here need the continuity for the uniform operator topology.
Theorems~13.5.8 and 13.7.3 of \cite{Hille1957} show that 
strongly continuous semigroups have such a continuity property uniformly with respect to $t$
in a compact interval of $(0,\infty)$.
These results are not quite sufficient for our purpose, 
because the continuity property in the interval $[0,\tau]$ is 
required in Proposition~\ref{lem:semigroup_epsilon}.
The main results in Sections~3 and 4,
Theorems~\ref{thm:MV_perturbation_T} and \ref{thm:analytic_T_diff},
can be regarded as extensions of
Theorems~13.5.8 and 13.7.3 of \cite{Hille1957}, respectively.


\section{General semigroup and Miyadera-Voigt perturbations}
\label{sec:general_case}
The objective of this section is to prove Theorem~\ref{thm:DT_case}.
We first recall the perturbation theorem of Miyadera-Voigt. This theorem
guarantees that for the class of perturbations $D$ in Theorem~\ref{thm:DT_case},
the perturbed operator $A+D$ generates a strongly continuous semigroup.
\begin{theorem}[\cite{Miyadera1966}, Corollary~III.3.16 of \cite{Engel2000}]
	\label{thm:MV_perturbation_case}
	Let $A$ be the generator of a strongly continuous semigroup $T(t)$ on a Banach space $X$, and
	let $D\in \mathcal{L}(X_1,X)$ satisfy
	\begin{equation}
	\label{eq:MV_perturvation}
	\int^{t_0}_0 \|DT(s)x\| ds \leq q \|x\|\qquad \forall x \in \dom(A)
	\end{equation}
	for some $0\leq q < 1$ and $t_0 >0$.
	Then the sum $A+D$ with domain $\dom(A+D) := \dom(A)$ generates a strongly continuous
	semigroup $T_D(t)$ on X. Moreover, 
	for every $x \in \dom(A)$ and every $t \geq 0$,
	$T_D(t)$ satisfies
	\begin{subequations}
		\begin{align}
		&T_D(t) x = T(t)x + \int^t_0 T(t-s)DT_D(s)x ds, \label{eq:TD_rep}\\
		&\int^{t_0}_0 \|DT_D(s)x\| ds \leq \frac{q}{1-q} \|x\|. \label{eqDTD_bound}
		\end{align}
	\end{subequations}
\end{theorem}
We call \eqref{eq:TD_rep} the {\em variation of parameter formula} for
the perturbed semigroup $T_D(t)$.

Using Theorem~\ref{sec:general_case}, we obtain the following result:
\begin{theorem}
	\label{thm:MV_perturbation_T}
	Let $A$ be the generator of a strongly continuous semigroup $T(t)$ on a Banach space $X$, and
	choose $t_0 > 0$ arbitrarily.
	For every $\epsilon,\tau>0$,
	there exists $0< q < 1$ such that for every $D\in \mathcal{L}(X_1,X)$ satisfying \eqref{eq:MV_perturvation},
	the perturbed semigroup  $T_D(t)$ generated by $A+D$ satisfies
	\begin{align}
	\label{eq:MV_perturbation_T}
	\|T(t) - T_D(t)\| < \epsilon \qquad \forall t \in [0,\tau ].
	\end{align}
\end{theorem}
\begin{proof}
	Let $t_0,\tau>0$ be given, and suppose that $D\in \mathcal{L}(X_1,X)$ satisfies \eqref{eq:MV_perturvation}
	for some $0< q < 1$.
	By \eqref{eq:TD_rep}, we obtain
	\begin{equation}
	\label{eq:T-TD_diff}
	T(t)x - T_D(t)x = \int^{t}_0T(t-s)DT_D(s)x ds \qquad \forall x \in \dom(A),~ \forall t\geq 0.
	\end{equation}
	
	Since every strongly continuous semigroup is uniformly bounded on
	a compact interval,
	there exists $M\geq 1$ such that $\|T(t)\| \leq M$ for every $t \in [0,\tau]$.
	Let $nt_0 \leq \tau < (n+1)t_0$ with $n \in \mathbb{Z}_+$.
	In the case $n=0$, we obtain $\tau < t_0$. Therefore, \eqref{eqDTD_bound} and \eqref{eq:T-TD_diff} yield
	\begin{align*}
	\|T(t)x - T_D(t)x \| &\leq 
	M\int^{t}_0\|DT_D(s)x\| ds 
	\leq M\int^{t_0}_0\|DT_D(s)x\| ds \\
	&\leq \frac{Mq}{1-q} \|x\|\qquad \forall x \in \dom(A),~\forall t \in [0,\tau].
	\end{align*}
	Since $\dom (A)$ is dense in $X$, we obtain
	\begin{equation}
	\label{eq:T_TD_diff1}
	\|T(t)- T_D(t) \| \leq \frac{Mq}{1-q} \qquad \forall t \in [0,\tau].
	\end{equation}
	
	Next consider the case $n \geq 1$. Similarly to the case $n=0$, we have
	\[
	\|T(t)- T_D(t) \| \leq \frac{Mq}{1-q}=:q_0\qquad \forall t \in [0,t_0].
	\]
	For every $k\in \mathbb{N}$ and every $t \in [0,t_0]$,
	\begin{align*}
	&\|T(kt_0+t) - T_D(kt_0+t)\| \\
	&\qquad \leq 
	\|T(kt_0+t) - T(t)T_D(kt_0)\| + \|T(t)T_D(kt_0)- T_D(kt_0+t)\| \\
	&\qquad \leq
	\|T(t)\| \cdot \|T(kt_0) - T_D(kt_0)\| + \|
	T_D(kt_0)
	\| \cdot \| 
	T(t)- T_D(t)
	\|.
	\end{align*}
	Hence, 
	for every $k \in \{1,\dots,n\}$,
	if $q_{k-1}>0$ satisfies
	\[
	\big\|T \big((k-1)t_0 +t \big) - T_D\big((k-1)t_0 +t \big) \| \leq q_{k-1}\qquad \forall t \in [0,t_0],
	\]
	then
	\[
	\|T(kt_0+t) - T_D(kt_0+t)\| \leq(M+q_0)q_{k-1} + Mq_0\qquad \forall t \in [0,t_0].
	\]
	Consider the sequence $\{q_k\}_{k \in \mathbb{Z}_+}$ constructed by
	\[
	q_k = (M+q_0)q_{k-1} +Mq_0\quad \forall k \in \mathbb{N};\qquad
	q_0 = \frac{Mq}{1-q}.
	\]
	Then $\{q_k\}_{k \in \mathbb{Z}_+}$ is increasing and satisfies
	\[
	\|T(kt_0+t) - T_D(kt_0+t)\| \leq q_k\qquad \forall t \in [0,t_0],~
	\forall k \in \{0,\dots,n  \}.
	\]
	Hence, we obtain
	$	\|T(t)- T_D(t) \| \leq q_n$
	for all $t \in [0,\tau]$.
	By definition, $q_n$ is continuous with respect to $q \in (0,1)$ and decreases to 0 as $q \to 0$.
	Thus for every $\epsilon,\tau >0$, there exists $0 <q < 1$ such that 
	\eqref{eq:MV_perturbation_T} holds.
	This completes the proof.
\end{proof}
From Theorem~\ref{thm:MV_perturbation_T}, 
we can easily obtain the following result on bounded perturbations.
\begin{corollary}
	Let $A$ be the generator of a strongly continuous semigroup $T(t)$ on a Banach space $X$.
	For every $\epsilon,\tau>0$,
	there exists $d >0$ such that for every
	$D \in \mathcal{L}(X,X)$ satisfying $\|D\|\leq d$, the perturbed semigroup  $T_D(t)$ generated by $A+D$ satisfies
	$\|T(t) - T_D(t)\| < \epsilon$ \hspace{-1pt} for all $t \in [0,\tau ]$.
\end{corollary}
\begin{proof}
	By the strong continuity of $T$, there exists $M\geq 1$ such that 
	$\|T(t)\| \leq M$ for every $t \in [0,\tau]$. For every $q \in (0,1)$ and 
	every $t_0 >0$,  $\|D\| \leq q/(t_0 M)$ yields \eqref{eq:MV_perturvation}.
	Therefore, the assertion  follows from
	Theorem~\ref{thm:MV_perturbation_T}. 
\end{proof}

\begin{proof}[Proof of Theorem \ref{thm:DT_case}]
	Suppose that the sampled-data system \eqref{eq:plant}
	is exponentially stable with decay rate greater than $\omega$.
	Lemma~\ref{lem:ex_po_stability} shows that $e^{\omega\tau}\Delta_\tau$ 
	is power stable.
	For the strongly continuous semigroup $T_D$ generated by $A+D$,
	we define 
	\[
	\Delta_{\tau,D} := T_D(\tau) + \int^\tau_0 T_D(s)B ds.
	\]
	Let $t_0 >0$ be given.
	By Proposition~\ref{lem:semigroup_epsilon} and Theorem~\ref{thm:MV_perturbation_T},
	there exists $0< q < 1$ such that $e^{\omega\tau}\Delta_{\tau,D}$ is power stable
	for every $D\in \mathcal{L}(X_1,X)$ satisfying \eqref{eq:MV_perturvation}.
	Using Lemma~\ref{lem:ex_po_stability} again, the perturbed sampled-data system  \eqref{eq:perturbed_plant} is also
	exponentially stable with decay rate greater than $\omega$.
\end{proof}

\section{Analytic semigroup and relatively bounded perturbations}
We next prove Theorem~\ref{thm:RB_case}.
In this theorem, we consider the following class of perturbations:
\begin{definition}
	\label{def:relatively_bounded_perturb}
	{\em
		Let $A:\dom (A) \subset X \to X$ be a linear operator on the Banach space $X$.
		A linear operator $D:\dom(D) \subset X \to X$ is called 
		a {\em relatively $A$-bounded perturbation} with constants $\alpha, \beta \geq 0$
		if 
		$\dom (A) \subset \dom (D)$ and if
		\begin{equation}
		\label{eq:relatively_bounded}
		\|Dx\| \leq \alpha \|Ax\| + \beta\|x\|\qquad \forall x \in \dom (A).
		\end{equation}
		Moreover,
		let $\mathcal{P}_{\alpha,~\!\beta}(A)$ represent the set of
		all relatively $A$-bounded perturbations with constants $\alpha, \beta \geq 0$.
	}
\end{definition}

In Definition~\ref{def:relatively_bounded_perturb}, we deal with a larger class of
perturbations than in Section~\ref{sec:general_case}, which is illustrated in
Example \ref{ex:multiplicative}  at the end of this section.

Let us assume that $A$ generates an analytic semigroup. We refer the readers to
Chapter~XVII in \cite{Hille1957},
Section~II.4.a in \cite{Engel2000}, and Section~3.7 in \cite{Arendt2001} for the definition and properties 
of analytic semigroups.
The following theorem guarantees that the perturbed operator $A+D$
also generates an analytic semigroup if $D$ is a relatively $A$-bounded perturbation
with sufficiently small constants.
\begin{theorem}[Theorem 13.7.1 of \cite{Hille1957},  Theorem~III.2.10 of \cite{Engel2000}]
	\label{thm:perturbed_analytic_semigroup}
	Let $A$ generate an analytic semigroup of angle $\delta \in (0,\pi/2]$ on a Banach space.
	There exists $\alpha,\beta > 0$ such that the sum
	$A+D$ with domain 
	$\dom(A+D) := \dom(A)$
	generates an analytic semigroup of angle at least $\delta$ for every $D \in \mathcal{P}_{\alpha,~\!\beta}(A)$.
\end{theorem}

The next result shows that if  the original semigroup is analytic and
if a relatively $A$-bounded 
perturbation satisfies \eqref{eq:relatively_bounded} with sufficiently small constants $\alpha, \beta >0$, 
then the difference between the original and perturbed semigroups is also small.
\begin{theorem}
	\label{thm:analytic_T_diff}
	Suppose that $A$ generates an analytic semigroup $T(z)$ of angle $\delta \in (0,\pi/2]$.
	For every $\epsilon,r>0$ and every $0<\delta_1<\delta$, there 
	exist constants $\alpha,\beta > 0$ such that 
	for every $D \in \mathcal{P}_{\alpha,~\!\beta}(A)$, the perturbed semigroup $T_D(z)$ generated by 
	$A+D$ satisfies
	\begin{equation}
	\label{eq:analytic_T_diff}
	\|T(z) - T_D(z)\| < \epsilon\qquad 
	\forall z \in \Sigma_{\delta_1} \cap \mathbb{D}_r.
	\end{equation}
\end{theorem}
\begin{proof}
	First we assume that $T(z)$ is bounded, i.e.,
	$\|T(z)\|$ is bounded in $\Sigma_{\delta'}$ for every $0 < \delta' < \delta$.
	For every $\lambda \in \varrho(A)$ and
	every relatively $A$-bounded perturbation $D$,
	we obtain
	\[
	\lambda-A-D = \big[I - DR(\lambda,A)\big] (\lambda I- A),
	\]
	and hence
	if $\|DR(\lambda,A)\| < 1$, then $\lambda \in \varrho(A+D)$ and
	\begin{equation}
	\label{eq:resolvent_AD}
	R(\lambda,A+D) = R(\lambda,A)\big[I - DR(\lambda,A)\big]^{-1}.
	\end{equation}
	
	Choose $\nu \in (0,1)$, 
	$\delta_1 \in (0,\delta)$, $\delta_2 \in (\delta_1,\delta)$, and $r >0$ arbitrarily.
	One can show that 
	there exist $\alpha,\beta > 0$ such that 
	\begin{equation}
	\label{eq:DR_epsilon}
	\|DR(\lambda,A)\| \leq \nu \qquad 
	\forall D \in \mathcal{P}_{\alpha,~\!\beta}(A),~
	\forall \lambda \in 
	\overline{\Sigma}_{\pi/2+\delta_2} \cap
	\overline{\mathbb{E}}_{1/r} =:\Sigma,
	\end{equation}
	using the standard techniques as in
	the proofs of Lemmas~III.2.5, III.2.6 and Theorem~III.2.10 of \cite{Engel2000}.
	Let $\alpha,\beta >0$ satisfy 
	\eqref{eq:DR_epsilon}. 
	By Theorem~II.4.6 in \cite{Engel2000},
	there exists $C  = C(\delta_2)>0$ such that
	for every $\lambda \in \overline{\Sigma}_{\pi/2+\delta_2} \setminus \{0\}$,
	we obtain $\lambda \in \rho(A)$ and 
	\begin{equation}
	\label{eq:resol_analytic_bound}
	\|R(\lambda,A)\| \leq \frac{C}{|\lambda|}.
	\end{equation}
	Therefore, 
	$\Sigma \subset \varrho(A+D) $ 
	for every 
	$D \in \mathcal{P}_{\alpha,~\!\beta}(A)$, and 
	\eqref{eq:resolvent_AD} yield
	\begin{equation}
	\label{eq:resolvent_bounded}
	\|R(\lambda,A+D)\| \leq \frac{C}{(1-\nu)|\lambda|}\qquad
	\forall D \in \mathcal{P}_{\alpha,~\!\beta}(A),~\forall \lambda \in \Sigma
	\end{equation}
	This inequality implies that 
	for every $D \in \mathcal{P}_{\alpha,~\!\beta}(A)$,
	$A+D$ generates an analytic semigroup of angle at least $\delta$; see, e.g., 
	Exercise~III.2.18 (3) of \cite{Engel2000}. Moreover, 
	since
	\begin{align*}
	R(\lambda,A) - R(\lambda,A+D) &= 
	R(\lambda,A) \big(
	I - \big[I - DR(\lambda,A)\big]^{-1}
	\big) \\
	&=
	-R(\lambda,A) DR(\lambda,A) \big[I - DR(\lambda,A)\big]^{-1},
	\end{align*}
	it follows 
	from \eqref{eq:resolvent_AD} and 
	\eqref{eq:resol_analytic_bound}
	that 
	\begin{equation}
	\label{eq:resol_diff}
	\|
	R(\lambda,A) -R(\lambda,A+D) 
	\| \leq
	\frac{C\nu }{(1-\nu )|\lambda|}\qquad
	\forall D \in \mathcal{P}_{\alpha,~\!\beta}(A),~\forall \lambda \in \Sigma
	\end{equation}
	
	Let $z \in \Sigma_{\delta_1} \cap \mathbb{D}_{r}$ and 
	$D \in \mathcal{P}_{\alpha,~\!\beta}(A)$ be given, where $\alpha,\beta > 0$ satisfy \eqref{eq:DR_epsilon}.
	Choose any piecewise smooth curve $\gamma$ in $\Sigma_{\pi/2 + \delta}$
	going from $\infty e^{-i (\pi/2 +\delta')}$ to $\infty e^{i (\pi/2 +\delta')}$ for some
	$\delta' \in (|\arg z|, \delta)$.
	Since $A+D$ generates an analytic semigroup $T_D$ of angle $\delta$ by
	the argument above,
	there exists $w \geq 0$ such that
	$A+D-wI$ generates a bounded analytic semigroup $e^{-wz}T_D(z)$ of angle $\delta$. Hence
	the integral representation of $T_D$ is given by (see, e.g., Definition~II.4.2 of \cite{Engel2000})
	\begin{equation}
	\label{eq:TD_IntRep}
	T_D(z) = \frac{1}{2\pi i}
	\int_{\gamma_{w}} e^{\mu z} R(\mu, A+D) d\mu,
	\end{equation}
	where 
	$\gamma_{w}$ is a piecewise smooth curve similar to $\gamma$ but shifted to the right by ${w}$, that is,
	$
	\gamma_{w}:= \{\mu\in \mathbb{C}: \mu = {w} + \mu_1,~ \mu_1 \in \gamma\}.
	$
	Consider an integral path $\gamma_z \subset \Sigma$ that consists of
	the three parts
	\begin{align*}
	&\gamma_{z,1} : 
	\{
	-\rho e^{-i(\pi/2+(\delta_1+\delta_2)/2)} :-\infty \leq \rho \leq -1/|z|
	\}, \\
	&\gamma_{z,2} : 
	\{
	e^{i \alpha }/|z| : - (\pi/2+(\delta_1+\delta_2)/2) \leq \alpha \leq \pi/2+(\delta_1+\delta_2)/2
	\},  \\
	&\gamma_{z,3} : 
	\{
	-\rho e^{i(\pi/2+(\delta_1+\delta_2)/2)} :1/|z| \leq \rho \leq \infty
	\}.
	\end{align*}
	Figure~\ref{fig:integral_path} illustrates the integral path $\gamma_z$.
	
\begin{figure}[tb]
	\centering
	\includegraphics[width = 5cm]{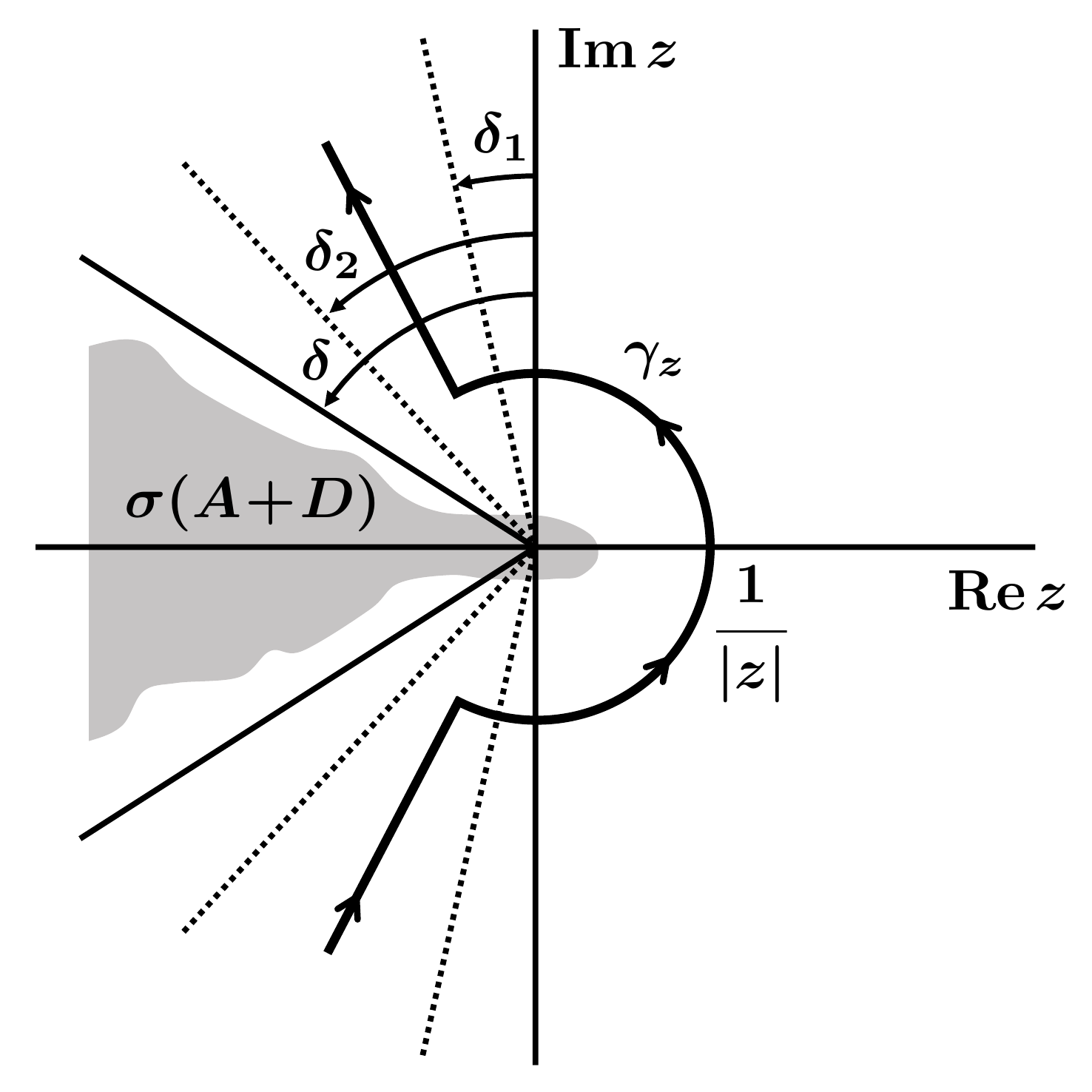}
	\caption{Integral path $\gamma_z$.}
	\label{fig:integral_path}
\end{figure}

	Since $\Sigma \subset \varrho(A+D)$, it follows from
	Cauchy's integral theorem and \eqref{eq:TD_IntRep}
	that
	\[
	T_D(z) = \frac{1}{2\pi i}
	\int_{\gamma_z} e^{\mu z} R(\mu, A+D) d\mu.
	\]
	The integral representation of $T$ is also given by
	\[
	T(z) = \frac{1}{2\pi i}
	\int_{\gamma_z} e^{\mu z} R(\mu, A) d\mu,
	\]
	and hence
	\[
	T(z) - T_D(z) 
	= 
	\frac{1}{2\pi i}
	\int_{\gamma_z} e^{\mu z} \big[R(\mu, A) - R(\mu, A+D) \big]d\mu.
	\]
	Hence \eqref{eq:resol_diff} yields
	\[
	\|T(z) - T_D(z)\|
	\leq
	\frac{C\nu }{2\pi(1-\nu) }
	\int_{\gamma_z}  \frac{e^{\re (\mu z)}}{|\mu|} d\mu\qquad
	\forall D \in \mathcal{P}_{\alpha,~\!\beta}(A).
	\]
	Since a calculation similar to that in the proof of
	Proposition~II.4.3 of \cite{Engel2000} shows that
	\[
	\frac{1}{2\pi}
	\int_{\gamma_z} \frac{e^{\re (\mu z)}}{|\mu|} d\mu
	\leq \frac{1}{\pi}
	\int^\infty_1 \frac{e^{-\rho \sin(\delta_2 - \delta_1)/2}}{\rho} d\rho + e
	=: L(\delta_1,\delta_2) < \infty,
	\]
	it follows that 
	\[
	\|T(z) - T_D(z)\|
	\leq
	\frac{CL(\delta_1,\delta_2)\nu }{1-\nu }\qquad
	\forall D \in \mathcal{P}_{\alpha,~\!\beta}(A).
	\]
	For every $\epsilon >0$, if we choose $\nu  \in (0,1)$ satisfying
	\[
	\nu  < \frac{\epsilon}{CL(\delta_1,\delta_2) + \epsilon},
	\]
	then \eqref{eq:analytic_T_diff} holds.
	
	To treat the general case, we take $w_0 \geq 0$ such that $A-w_0I$ generates
	a bounded analytic semigroup of angle $\delta$. Since
	\[
	\|Dx\| \leq \alpha \|Ax\| + \beta\|x\| \leq 
	\alpha \|(A-w_0I)x\| + (\alpha w_0 + \beta) \|x\|\qquad \forall x \in \dom (A),
	\]
	the operator
	$D$ is also relatively $(A-w_0I)$-bounded and the coefficients $\alpha$ and $\alpha w_0 + \beta$
	can be arbitrarily small for suitable $\alpha, \beta$.
	Take $z \in \Sigma_{\delta_1} \cap \mathbb{D}_{r}$.
	Replacing $A$ by $A-w_0I$ in the argument above, we obtain
	\begin{align*}
	\|e^{-w_0z}T(z) - e^{-w_0z}T_D(z)\|
	&\leq \frac{\nu L(\delta_1,\delta_2)}{2\pi(1-\nu )}\qquad
	\forall D \in \mathcal{P}_{\alpha,~\!\beta}(A).
	\end{align*}
	and hence
	\[
	\|T(z) - T_D(z)\| \leq \frac{e^{w_0r} C L(\delta_1,\delta_2) \nu }{1-\nu }\qquad
	\forall D \in \mathcal{P}_{\alpha,~\!\beta}(A).
	\]
	Thus,  for every $\epsilon >0$, if $\nu  \in (0,1)$ satisfies
	\[
	\nu  < \frac{\epsilon}{e^{w_0r}CL(\delta_1,\delta_2) + \epsilon},
	\]
	then we obtain \eqref{eq:analytic_T_diff}. This completes the proof.
\end{proof}

\begin{proof}[Proof of Theorem~\ref{thm:RB_case}]
	If we use Theorem~\ref{thm:analytic_T_diff} instead of Theorem~\ref{thm:MV_perturbation_T},
	then we obtain
	the proof of Theorem~\ref{thm:RB_case}, by proceeding as in the proof of Theorem~\ref{thm:DT_case}
	above. We therefore omit the detail. 
\end{proof}

We conclude this section with the comparison of the perturbation classes in 
Theorems \ref{thm:DT_case} and \ref{thm:RB_case}.
Let $A$ be a generator on a Banach space $X$.
If the perturbation $D$ is small in the sense of \eqref{eq:MV_perturvation}, then
$D$ is small in the sense of \eqref{eq:relatively_bounded} as well. More precisely,
for every $\alpha,\beta >0$, $t_0>0$, and $D \in \mathcal{L}(X_1,X)$, there exists $q>0$ such that
\eqref{eq:MV_perturvation} implies
\eqref{eq:relatively_bounded}; see, e.g., 
Exercise~III.2.18 (2) of \cite{Engel2000}. However, the converse is false as shown below.

\begin{example}
	\label{ex:multiplicative}
	{\em
		Here we provide a perturbation that 
		is small in the sense of \eqref{eq:relatively_bounded} but
		large in the sense of \eqref{eq:MV_perturvation}.
		Since the semigroup we consider in this example is not analytic,
		Theorem~\ref{thm:analytic_T_diff} is not applicable.
		We will therefore see that even if the perturbation is sufficiently small in the sense of 
		\eqref{eq:relatively_bounded}, 
		the difference between the original and perturbed semigroups
		is not small.

		Consider the Banach space \[X := \big\{
		(x_n)_{n \in \mathbb{N}}: x_n \in \mathbb{C}~~\text{for every $n \in \mathbb{N}$} \text{~~and~~}
		\lim_{n \to \infty} x_n = 0
		\big\}\]
		with norm $\|(x_n)_{n \in \mathbb{N}}\| := \sup_{n \in \mathbb{N}}|x_n|$.
		On the Banach space $X$, we 
		take the strongly continuous (semi)group $T(t)$ given by
		\[
		T(t)(x_n)_{n \in \mathbb{N}} = \big(e^{(-1+in)t}x_n\big)_{n \in \mathbb{N}}\qquad \forall t \geq 0.
		\]
		This semigroup $T(t)$ is generated by
		\[
		A(x_n)_{n \in \mathbb{N}} := \big((-1+in)x_n\big)_{n \in \mathbb{N}} 
		\]
		with domain
		\[
		\dom(A) := \big\{
		(x_n)_{n \in \mathbb{N}} \in X: \big((-1+in)x_n\big)_{n \in \mathbb{N}}  \in X
		\big\};
		\]
		see, e.g., 
		Lemma in Section II.2.9 of \cite{Engel2000}.
		The semigroup $T(t)$ is not analytic.
		
		Consider the perturbation operator $D_m:X \to X$, $m \in \mathbb{N}$, defined by
		\[
		D_m(x_n)_{n \in \mathbb{N}} := 
		(\underbrace{0,\dots,0}_{m-1},\sqrt{m}x_m,0,\dots).
		\]
		The operator $D_m$ is bounded for each $m \in \mathbb{N}$, but
		the norm of $D_m$ diverges as $m \to \infty$.
		
		\noindent
		{\sf (a)} First we investigate the conditions \eqref{eq:MV_perturvation} and 
		\eqref{eq:relatively_bounded} for $D_m$.
		If we take 
		\begin{equation}
		\label{eq:only_m}
		(z_n)_{n \in \mathbb{N}}  = (\underbrace{0,\dots,0}_{m-1},z_m,0,\dots) \in \dom(A),
		\end{equation}
		then every $t_0 >0$,
		\begin{align*}
		\int^{t_0}_0 \big\| D_mT(t)(z_n)_{n \in \mathbb{N}} \big\| dt 
		=
		\int^{t_0}_0 \big|\! \sqrt{m}e^{(-1+im)t} z_m \big| dt =
		\sqrt{m}\!~(1-e^{-t_0}) \|(z_n)_{n \in \mathbb{N}} \|.
		\end{align*}
		On the other hand, for every $(x_n)_{n \in \mathbb{N}} \in \dom(A)$,
		\[
		\big\|A (x_n)_{n \in \mathbb{N}}  \big\| \geq 
		|-1+im| \cdot |x_m|\geq m|x_m|,
		\]
		and hence
		\[
		\big\|D_m(x_n)_{n \in \mathbb{N}} \big\|
		= \sqrt{m} ~\! |x_m| \leq 
		\frac{\big\|A (x_n)_{n \in \mathbb{N}}  \big\|}{\sqrt{m} }  .
		\]
		Therefore, as $m \to \infty$, the perturbation $D_m$
		becomes larger in the sense of \eqref{eq:MV_perturvation} but
		smaller in the sense of \eqref{eq:relatively_bounded}.
		
		\noindent
		{\sf (b)} Let us next examine the difference  between
		the semigroup and the perturbed semigroup.
		Since $D_m$ is bounded,
		it follows that each $A+D_m$ is the generator of the strongly continuous semigroup $T_m(t)$, which
		is given by
		\begin{align*}
		T_m(t) (x_n)_{n \in \mathbb{N}}
		&=
		(
		e^{(-1+i)t}x_1,\dots,e^{(-1+i(m-1))t} x_{m-1}, \\
		&\qquad \qquad e^{(-1+im)t+\sqrt{m}t}x_m, e^{(-1+i(m+1))t} x_{m+1},\dots
		)
		\end{align*}
		for every $t\geq 0$.
		Hence, $(z_n)_{n \in \mathbb{N}}$ in \eqref{eq:only_m} satisfies
		\begin{align*}
		\big\|
		T(t)(z_n)_{n \in \mathbb{N}} - T_m(t)(z_n)_{n \in \mathbb{N}}
		\big\|
		&=
		\big|e^{(-1+im)t}z_m - e^{(-1+im)t+\sqrt{m}t}z_m \big|  \\
		&= e^{-t}(e^{\sqrt{m}t} - 1)\big\|(z_n)_{n \in \mathbb{N}}\big\|\qquad \forall t\geq 0,
		\end{align*}
		which implies that $\|T(t) - T_m(t)\| \geq e^{-t}(e^{\sqrt{m}t} - 1)$ for every $m\in \mathbb{N}$ and
		every $t \geq 0$.
		From this example, we find that 
		even if the perturbation is sufficiently small in the sense of \eqref{eq:relatively_bounded},
		there exists a strongly continuous semigroup such that the difference (in the uniform operator topology) between
		the original and perturbed semigroups is not small.
		Moreover, a similar calculation shows that $\|T_m(t)\| \geq e^{(\!\sqrt{m} - 1)t}$ for every $m \in \mathbb{N}$ 
		and $t\geq 0$. 
		Since $\|T(t)\| = e^{-t}$, the open-loop system $\dot x = Ax$ is exponentially stable.
		However, the perturbed open-loop system $\dot x = (A+D_m)x$ is not exponentially stable for any
		$m \in \mathbb{N}$, although
		$D_m$ becomes sufficiently small in the sense of \eqref{eq:relatively_bounded} as $m \to \infty$.
		
	}
\end{example}

%
%
%
%
%
\section{Examples}
In this section, we illustrate
the proposed robustness analysis with
delayed perturbations and low-order differential perturbations.
First, we apply Theorem~\ref{thm:DT_case} to delayed perturbations, 
employing the techniques developed in Chapter~3 in \cite{Batkai2005}.
To analyze the robust stability of diffusion systems, we exploit
Theorem~\ref{thm:RB_case} together with 
the results obtained in Example III.2.12 of \cite{Engel2000} and 
Example 3.7.34 of \cite{Arendt2001}.

Before presenting examples,
we recall the definitions of vector-valued function spaces.
For an interval $J \subset \mathbb{R}$, we define the vector-valued Lebesgue space $L^2(J,X)$
and
Sobolev space $H^1(J,X)$ by
\begin{align*}
L^2(J,X) &:= \left\{
f:J \to X: \text{$f$ is measurable and~} \int_J\|f(s)\|^2 dx < \infty  
\right\}, \\
H^1(J,X) &:= \left\{
f \in L^2(J,X): 
\exists s_0\in J,~g \in L^2(J,X) \text{~s.t.~}
f(s) = f(s_0) + \int^s_{s_0} g(s)ds
\right\}.
\end{align*}
The spaces  $L^2(J,X)$ and $H^1(J,X)$ are endowed, respectively, with
the norms
\[
\|f\|_{L^2} :=
\sqrt{\int_J\|f(s)\|^2 dx},\quad
\|f\|_{H^1} := \|f\|_{L^2} + \|g\|_{L^2},
\]
where $g$ satisfies the condition in the definition of $H^1(J,X)$.
We set $L^2(J) := L^2(J,\mathbb{C})$ and $H^1(J) := H^1(J,\mathbb{C})$.
Similarly, 
$L^p(J)$, $p \in [1,\infty]$, and
$H^2(J)$ are denoted by the standard 
scalar-valued Lebesgue space and
second-order Sobolev space, respectively.
\begin{example}[Delayed perturbations]
	{\em
		For a Banach space $X$,
		define a bounded linear operator $\Phi:H^1\big( 
		[-1,0],X
		\big) \to X$ in the form of a Riemann-Stieltjes integral:
		\begin{equation}
		\label{eq:Phi_def}
		\Phi(f) := \int^0_{-1}d\eta(s)f(s) ,
		\end{equation}
		where $\eta:[-1,0] \to \mathcal{L}(X)$ is of bounded variation.
		We refer the readers to Chapter~2 of \cite{Dudley2011} for the Riemann-Stieltjes integral of a vector-valued function.
		For $k\in \{1,\dots,n\}$,
		let
		$D_k \in \mathcal{L}(X)$,
		$h_k \in [0,1]$, and
		$\mathds{1}_{[-h_k,0]}$ be the characteristic function of $[-h_k,0]$.
		If we set $\eta:= \sum_{k=1}^nD_k \mathds{1}_{[-h_k,0]}$, then $\Phi$ is the discrete delay operator given by
		\[
		\Phi(f) = \sum_{k=1}^n D_kf(-h_k).
		\]

		For a function $z: [-1,\infty) \to X$ and $t\geq0$, we define $z_t:[-1,0] \to X$ by
		$
		z_t(\sigma) := z(t+\sigma).
		$
		Let 
		$A$ be the generator of a strongly continuous semigroup $T(t)$ on $X$. 
		For a Banach space $U$, let
		$B_0 \in \mathcal{L}(U,X)$ and
		$F \in \mathcal{L}(X,U)$.
		Consider the following sampled-data system with delayed perturbation:
		\begin{subequations}
			\label{eq:delayed_equation}
			\begin{align}
			&\dot z(t) = Az(t) + \Phi z_t + B_0u(t) \quad t\geq 0;\qquad 
			z(0) = \zeta,~z_0 = f, \\
			&u(k\tau + t) = F z(k\tau)\qquad t \in[0, \tau),~k \in \mathbb{Z}_+.
			\end{align}
		\end{subequations}
		where $\zeta \in X$, $f \in L^2\big([-1,0],X\big)$, and $\tau >0$.
		Define $\mathcal{X} := X \times L^2\big([-1,0],X\big)$ with norm
		\[
		\left\|~\!
		\begin{bmatrix}
		z \\ f
		\end{bmatrix}
		~\!
		\right\|_{\mathcal{X}} := \|z\|_X + \|f\|_{L^2},\quad
		\text{where $\|\cdot\|_X$ is the norm in $X$}.
		\]
		Set
		\[
		\mathcal{A} := 
		\begin{bmatrix}
		A & 0 \\
		0 & d/d\theta
		\end{bmatrix}\quad
		\text{with~}
		\dom(\mathcal{A}) := 
		\left\{
		\begin{bmatrix}
		\zeta \\ f
		\end{bmatrix}
		\in \dom (A) \times H^1\big([-1,0],X\big):f(0) = \zeta
		\right\},
		\]
		where  $d/d\theta$ denotes the distributional derivative.
		We denote by $\|\cdot\|_\mathcal{A}$ 
		the graph norm for $\mathcal{A}$, i.e., 
		$\|x\|_{\mathcal{A}} := \|x\|_{\mathcal{X}} + \|\mathcal{A} x\|_{\mathcal{X}}$ for 
		$x \in \dom(\mathcal{A})$,
		and set $\mathcal{X}_1 := \big(\hspace{-1pt}\dom (\mathcal{A}), \|\cdot\|_\mathcal{A}\big)$.
		Define 
		\begin{align*}
		\mathcal{D} &:=
		\begin{bmatrix}
		0 & \Phi \\
		0 & 0
		\end{bmatrix} \in \mathcal{L}(\mathcal{X}_1,\mathcal{X}),\quad
		\mathcal{B} := 
		\begin{bmatrix}
		B \\ 0
		\end{bmatrix} \in \mathcal{L}(U,\mathcal{X}),\quad
		\mathcal{F} := 
		\begin{bmatrix}
		F & 0
		\end{bmatrix} \in \mathcal{L}(\mathcal{X},U).
		\end{align*}
		
		Setting 
		\[
		x(t) := \begin{bmatrix}
		z(t) \\ z_t
		\end{bmatrix},
		\]
		we find that 
		\eqref{eq:delayed_equation} can be written in the abstract form
		\begin{align*}
		&\dot x(t) = (\mathcal{A} + \mathcal{D}) x+ 
		\mathcal{B} u(t)\quad t \geq 0;\qquad
		x(0) = \begin{bmatrix}
		\zeta \\ f
		\end{bmatrix}\\
		&u(k\tau + t) = \mathcal{F} x(k\tau)\qquad t \in[0, \tau),~k \in \mathbb{Z}_+.
		\end{align*}
		see Section~3.1 of \cite{Batkai2005} for details. Moreover,
		$\mathcal{A}$ generates a strongly continuous semigroup $\mathcal{T}(t)$ by
		Theorem~3.2.5 of \cite{Batkai2005}.
		For every $t_0 \in (0,1)$,
		\[
		\int^{t_0}_0 \left\|
		\mathcal{D} \mathcal{T}(s)
		x
		\right\|_{\mathcal{X}} ds \leq
		t_0^{1/2}\cdot  \sup_{t \in [0,1]} \|T(t)\| \cdot v(\eta;[-1,0]) \cdot \|x\|_{\mathcal{X}}  \qquad \forall x \in \dom(\mathcal{A})
		\]
		by (3.41) of \cite{Batkai2005},
		where $v(\eta;[-1,0])$ is the total variation of $\eta$ on $[-1,0]$.
		By
		Theorem~\ref{thm:DT_case},
		if the total variation $v(\eta;[-1,0])$ is sufficiently small and
		if the nominal sampled-data system is exponentially stable  with decay rate greater than
		$\omega$,
		then 
		the perturbed sampled-data system 
		is also exponentially stable with decay rate greater than
		$\omega$.
	}
\end{example}
\begin{example}[Perturbations by general low-order differential operators]
	{\em
		Consider the second-order derivative operator $A$ on $X := L^2[0,1]$
		\[
		Af := \frac{d^2f}{d\xi^2}\quad
		\text{with~} \dom(A) := \big\{
		f \in H^2[0,1]: f(0) = f(1)
		\big\}.
		\]
		This operator generates an analytic semigroup on $X$, and
		the first-order derivative $\Gamma_0 := d/d\xi$ with maximal domain $H^1[0,1]$ is relatively 
		$A$-bounded, i.e.,
		there exist constants $\alpha_0,\beta_0 \geq0$ such that
		\[
		\|\Gamma_0f\|_{L^2} \leq \alpha_0 \|Af\|_{L^2} + \beta_0 \|f\|_{L^2}\qquad \forall f \in \dom (A);
		\]
		see Example~III.2.12 (i) of \cite{Engel2000}.
		Take $D \in \mathcal{L}\big(H^1[0,1], L^2[0,1]\big)$. Using Poincare's inequality
		(see, e.g., Sec.~1.7.2  in \cite{Dym1972}),
		\[
		\|f\|_{L^2} \leq \frac{1}{\pi} \|\Gamma_0 f\|_{L^2}\qquad \forall f \in \dom (A),
		\]
		we have that
		\begin{align*}
		\|Df\|_{L^2} &\leq \|D\|\cdot \|f\|_{H^1} 
		\leq \frac{1+\pi}{\pi} \|D\| \cdot \left\|\Gamma_0f \right\|_{L^2} \\
		&\leq \frac{1+\pi}{\pi}\|D\| ~\! \big(
		\alpha_0 \|Ax\|_{L^2} + \beta_0 \|x\|_{L^2}
		\big)\qquad \forall f \in \dom (A),
		\end{align*}
		By Theorem~\ref{thm:RB_case},
		if $\|D\|$ is sufficiently small and 
		if the nominal sampled-data system with generator $A$ is exponentially stable with decay rate greater than
		$\omega$, then
		the sampled-data system perturbed by $D$ is also exponential stability  with decay rate greater than
		$\omega$.
	}
\end{example}

\begin{example}[Perturbations by low-order differential operators and multiplication operators]
	{\em 
		We consider the $n$-dimensional diffusion semigroup on $X:=L^1(\mathbb{R}^n)$
		given by
		\[
		(T(t)f)(x) := (4\pi t)^{-n/2} 
		\int_{\mathbb{R}^n} e^{-|x-r|^2/(4t)} f(r)dr
		\]
		for every $t >0$ and every $x \in \mathbb{R}^n$.
		Let us denote by $A$ the generator of $T(t)$.
		By Proposition in Paragraph~II.2.13 of \cite{Engel2000},
		the generator $A$ 
		is the closure of the Laplacian defined on the Schwartz space, i.e., 
		the space of rapidly decreasing, infinitely differentiable functions.
		For $j \in \{1,\dots,n \}$, let $U_j$ be the strongly continuous group on $X$ defined by
		\[
		(U_j(t) f)(x) := f(x_1,\dots,x_{j-1}, x_j+t,\dots,x_n)
		\]
		for every $t \in \mathbb{R}$ and every $x \in \mathbb{R}^n$. For $j \in \{1,\dots,n \}$,
		and 
		the generator of $U_j$ be denoted by $\Gamma_j$, which
		is a first-order differential operator.
		Let $d_j \in L^{\infty}(\mathbb{R}^n)$ and define 
		\[
		(D_1 f)(x) := \sum_{j=1}^n d_j(x) (\Gamma_jf)(x)
		\] 
		for every $x \in \mathbb{R}^n$ with domain \[\dom (D_1) := \big\{
		f \in L^1(\mathbb{R}^n): 
		\text{$f \in \dom(\Gamma_j)$ for $j \in \{1,\dots,n\}$ with $d_j \not \equiv 0$}
		\big\}.
		\]
		Let $p > \max\{1,n/2\}$ and $q \in L^p(\mathbb{R}^n)$, and define
		\[
		(D_2f)(x) := q(x) f(x)
		\]
		for every $x \in \mathbb{R}^n$ with domain  $\dom (D_2) := 
		\big\{
		f \in L^1(\mathbb{R}^n): qf \in L^1(\mathbb{R}^n)
		\big\}$.
		It is shown in Example 3.7.34 of \cite{Arendt2001} that 
		$D_1$ is relatively $A$-bounded and satisfies
		\begin{equation}
		\label{eq:D1_bound}
		\|D_1f\|_{L^1} \leq
		\sum_{j=1}^n \|d_j\|_{L^{\infty}}\big(
		\|Af\|_{L^1}+ \|f\|_{L^1}
		\big)\qquad \forall f \in \dom (A).
		\end{equation}
		As shown in Example III.2.12 (ii) of \cite{Engel2000}, $D_2$ is also relatively $A$-bounded and satisfies
		\begin{equation}
		\label{eq:D2_bound}
		\|D_2f\|_{L^1} \leq
		c \|q\|_{L^p}
		\big(
		\|Af\|_{L^1} + \|f\|_{L^1}
		\big) \qquad \forall f \in \dom(A),
		\end{equation}
		where the constant $c>0$ depends only on $p$ and $n$.
		Together with \eqref{eq:D1_bound} and \eqref{eq:D2_bound}, Theorem~\ref{thm:RB_case} shows that
		if $\|d_1\|_{L^{\infty}},\dots, \|d_n\|_{L^{\infty}}$ and $\|q\|_{L^p}$ are sufficiently small and
		if the nominal sampled-data system with generator $A$ is exponentially stable with decay rate greater than
		$\omega$, then
		the sampled-data system perturbed by $D:= D_1 + D_2$ is also exponentially stable with decay rate greater than
		$\omega$.
	}
\end{example}

\section{Concluding remarks}
We have analyzed the exponential stability of infinite-dimensional  sampled-data systems
under unbounded perturbations.
First we have showed that exponential stability is preserved 
if the perturbed semigroup is sufficiently close to the original semigroup in the uniform
operator topology.
Then we have presented two results on the continuity of strongly continuous semigroups
with respect to their generators.
The first result is based on the perturbation theorem of Miyadera-Voigt.
The second result is obtained from the results on analytic semigroups and relatively bounded perturbations.
In this study, we have considered only bounded control operators in this study, and 
future work will involve extending the obtained results to systems with 
unbounded control operators.

\end{document}